%% file: ACOPF_DIGS.tex
\documentclass[10pt,journal,a4paper]{article}

\usepackage{graphicx}
\usepackage{algorithm}
\usepackage{amsmath}
\usepackage{amsfonts}
\usepackage{amssymb}
\usepackage{amsthm}
\usepackage{algorithmic}
\usepackage{color}
\usepackage{subfig}
\usepackage{tabularx}

\usepackage{soul}
\usepackage{multicol}
\usepackage{geometry}
\newtheorem{example}{Example}

\newtheorem{remark}{Remark}
\newtheorem{theorem}{Theorem}
\newtheorem{proposition}{Proposition}
\newtheorem{assumption}{Assumption}
\newcommand{\stt}{\text{ s.t. }}
\newcommand\withdvi[1]{#1} 

\newcommand\Sos{{\operatorname{\text{\small $\Sigma$}}}}

\allowdisplaybreaks
\newcommand{\inTR}[1]{}
\title{Optimal Power Flow \\ as a Polynomial Optimization Problem}

\author{Bissan Ghaddar, Jakub Marecek, Martin Mevissen
\thanks{IBM Research - Ireland, Damastown Industrial Estate, Mulhuddart, Dublin 15.
email: bghaddar@ie.ibm.com, jakub.marecek@ie.ibm.com, martmevi@ie.ibm.com }
}

\begin{document}

\maketitle

\begin{abstract}
Formulating the alternating current optimal power flow (ACOPF) as a polynomial optimization problem makes it possible to solve large instances in practice and to guarantee asymptotic convergence in theory. 

We formulate the ACOPF as a degree-two polynomial program and study two approaches to solving it via convexifications. In the first approach, we tighten 
the first order relaxation of the non-convex quadratic program by adding valid inequalities. 
In the second approach, we exploit the structure of the 
polynomial program by using a sparse variant of Lasserre's hierarchy.
This allows us to solve instances of up to 39 buses to global optimality
and to provide strong bounds for the Polish network within an hour.
\end{abstract} 


\input{Intro.tex}

\input{OPF_Formulation.tex}

\input{LL_Duality.tex}
\input{PP_Dense.tex}

\input{PP_Cuts.tex}
\input{PP_Sparse.tex}
\input{results.tex}
\section{Conclusion}\label{sec:conclusion}
In this work, we propose to formulate the optimal power flow as a polynomial programming problem
problem and present two techniques for deriving relaxations stronger than that of Lavaei and Low \cite{lavaei2012zero}. 
For several instances on up to 39 buses, where the Lavaei-Low relaxation is not exact, we provide provide globally optimal solutions for the first time. 
Furthermore, we show that the strong relaxations are tractable for medium- and large-scale instances.

The approaches are rather general. They make it possible to use arbitrary multivariate polynomials in the objective function and constraints, 
without the need to change the solver. 
Additionally, binary variables can be included, e.g., to model discrete decisions in transmission switching. 

\bibliographystyle{abbrv}
\bibliography{paper_bib}

\clearpage
\newpage
\appendix
\begin{proof}[Proof of Theorem \ref{thm:equivalence}]
Notice that the variables in [OP$_2$] are $x$, $P^g_k$, and $P_{lm}$ and $Q_{lm}$. However, not all the monomials appear in the polynomial formulation and hence using the first level of the hierarchy, $\mathcal{{K}}^{r}_G$, where $r=2$ one can approximate [OP$_2$]:{\small
\begin{align*}
\begin{split}
\max \quad & \varphi \\
\stt & \sum_{k\in G} \left(c^2_k(P_k^g)^2+c^1_k(P_k^d + \text{tr}(Y_kxx^T))+c^0_k\right) - \varphi \\
&= A(x) + \sum_{k\in G}B_k(P^g_k)+ \sum_{(l,m) \in L}C_{lm}(P_{lm},Q_{lm}) \\
&+ \sum_{k \in N} \overline{\lambda}_k( P_k^{\max}-P_k^d -\text{tr}(Y_kxx^T)) \\ &
+ \sum_{k \in N} \underline{\lambda}_k( -P_k^{\min}+P_k^d +\text{tr}(Y_kxx^T)) \\
&+ \sum_{k \in N} \overline{\gamma}_k( Q_k^{\max}-Q_k^d -\text{tr}(\bar{Y}_kxx^T)) \\
&+ \sum_{k \in N} \underline{\gamma}_k( -Q_k^{\min}+Q_k^d +\text{tr}(\bar{Y}_kxx^T)) \\
& +\sum_{k \in N} \overline{\mu}_k( (V_k^{\max})^2-\text{tr}(M_kxx^T))  \\
&+ \sum_{k \in N} \underline{\mu}_k( (-V_k^{\min})^2+\text{tr}(M_kxx^T)) \\
& +\sum_{(l,m) \in L} a_{lm} ((S_{lm}^{\max})^2  -P_{lm}^2-Q_{lm}^2) \\
&+  \sum_{k \in G}b_k (P^g_k-  \text{tr}(Y_kxx^T)-P_k^d) \\
& +\sum_{(l,m) \in L} c_{lm} (P_{lm} - \text{tr}(Y_{lm}xx^T))\\
&+\sum_{(l,m) \in L} d_{lm} (Q_{lm} - \text{tr}(\bar{Y}_{lm}xx^T))
\end{split}
\end{align*}}
where $A(x), B_k(P^g_k), C_{lm}(P_{lm},Q_{lm})$ are polynomials that are sum of squares as a function of $x$, $P^g_k$, and $P_{lm}$ and $Q_{lm}$ respectively. That is $A(x)=xAx^T$, $B_k(P^g_k)=  \left[ \begin{matrix} 1 \\ P^g_k \end{matrix}  \right]B_k  \left[ \begin{matrix} 1 \\ P^g_k \end{matrix}  \right]^T,$ and $C_{lm}(P_{lm},Q_{lm})=  \left[ \begin{matrix} 1 \\ P_{lm} \\ Q_{lm} \end{matrix}  \right] C_{lm} \left[ \begin{matrix} 1 \\ P_{lm} \\ Q_{lm} \end{matrix}  \right]^T,$ where $A, B_k,$ and $C_{lm}$ are positive semidefinite matrices of dimension $2|N|\times 2|N|$, $2\times 2$ and $3\times 3$ respectively. The variables $\overline{\lambda}_k, \underline{\lambda}_k, \overline{\gamma}_k, \underline{\gamma}_k, \overline{\mu}_k, \underline{\mu}_k$, and $a_{lm}$,are non-negative variables and $b_k, c_{lm},$ and $d_{lm}$ are free variables. By equating the coefficients of the monomials of the above problem, we rewrite it as{\small
\begin{align*}
\begin{split}
\max  \quad& \varphi \\
\stt & \sum_{k \in G} c^1_kP^d_k +\sum_{k \in G} c^0_k - \varphi \\
&= \sum_{k\in G}B_k^{00}+ \sum_{(l,m) \in L}C_{lm}^{00} + \sum_{k \in N} \overline{\lambda}_k (P_k^{\max}-P_k^d )  \\
&  + \sum_{k \in N} \underline{\lambda}_k( -P_k^{\min}+P_k^d )+ \sum_{k \in N} \overline{\gamma}_k( Q_k^{\max}-Q_k^d ) \\
&+ \sum_{k \in N} \underline{\gamma}_k( -Q_k^{\min}+Q_k^d )   + \sum_{k \in N} \overline{\mu}_k( V_k^{\max})^2 \\
&- \sum_{k \in N} \underline{\mu}_k(V_k^{\min})^2 + \sum_{(l,m) \in L} a_{lm}(S_{lm}^{\max})^2- \sum_{k\in G}b_kP_k^d \\
 & \sum_{k \in N}c_1^k Y_k = A-\sum_{k \in N}\left(\overline{\lambda}_k Y_k +\underline{\lambda}_k Y_k -\overline{\gamma}_k \bar{Y}_k \right. \\& +\left.\underline{\gamma}_k \bar{Y}_k - \overline{\mu}_k (V_k^{\max})^2   +\underline{\mu}_k (V_k^{\min})^2 - b_k Y_k \right) \\&-\sum_{(l,m) \in L} \left( c_{lm} Y_{lm}  + d_{lm} \bar{Y}_{lm} \right)\\
 & 0 = 2B^{12}_k+b_k \qquad c_k^2 = B^{22}_k \\
 & 0 = c_{lm} +2C_{lm}^{12} \qquad 0 = d_{lm} +2C_{lm}^{13} \\
  & 0 = 2C_{lm}^{23} \qquad\qquad\quad  0=-a_{lm}+ C^{22}_{lm} \\& 0 = -a_{lm} +C^{33}_{lm}  \\
  & A, B_k, C_{lm} \succeq 0.
  \end{split}
\end{align*}}
By substituting some of the variables:{\small
\begin{align*}
\begin{split}
\textbf{[OP$_2$-H$_1$]$^*$ }\\
\max  \quad& \sum_{k \in G} c^1_kP^d_k +\sum_{k \in G} c^0_k -\sum_{k\in G}B_k^{00}- \sum_{(l,m) \in L}C_{lm}^{00} \\
&- \sum_{k \in N} \overline{\lambda}_k (P_k^{\max}-P_k^d )    - \sum_{k \in N} \underline{\lambda}_k( -P_k^{\min}+P_k^d )\\
&- \sum_{k \in N} \overline{\gamma}_k( Q_k^{\max}-Q_k^d ) - \sum_{k \in N} \underline{\gamma}_k( -Q_k^{\min}+Q_k^d ) \\
&  - \sum_{k \in N} \overline{\mu}_k( V_k^{\max})^2 +\sum_{k \in N} \underline{\mu}_k(V_k^{\min})^2 \\
&- \sum_{(l,m) \in L}  C^{22}_{lm}(S_{lm}^{\max})^2 -\sum_{k\in G} 2B^{12}_kP_k^d \\
\stt & A = \sum_{k \in N} \left(c_1^k Y_k +\overline{\lambda}_k Y_k -\underline{\lambda}_k Y_k +\overline{\gamma}_k \bar{Y}_k \right. \\
&\left. -\underline{\gamma}_k \bar{Y}_k+ \overline{\mu}_k (V_k^{\max})^2 -\underline{\mu}_k (V_k^{\min})^2 -2B^{12}_k Y_k \right) \\
&- \sum_{(l,m) \in L} \left(2C_{lm}^{12} Y_{lm}  +2C_{lm}^{13} \bar{Y}_{lm}\right) \\
  &c_k^2 = B^{22}_k \qquad 0 = 2C_{lm}^{23} \\
  & C^{22}_{lm}- C^{33}_{lm} =0 \\
  & A, B_k, C_{lm} \succeq 0.
  \end{split}
\end{align*}}
which is equivalent to optimization problem 4 described in \cite{lavaei2012zero}, i.e., the dual of [OP-SDP].
\end{proof}

\end{document}

%% file: Intro.tex
\section{Introduction}

Optimal Power Flow (OPF) in alternating current models (ACOPF) is one of the most important 
power system optimization problems. Various optimization methods have been widely used to tackle this hard problem \cite{LavaeiTZ2014,zhang2013}. There are numerous extensions of the problem of widely varying tractability, including security-constrained variants taking into account uncertainty \cite{capitanescu2011}. Even the ACOPF alone, however, is a large-scale non-convex non-linear optimization problem, and hence challenging to solve.

While non-linear formulations for OPF capture the system behavior more accurately than linearization, in principle, they pose a challenge for the solvers, which 
often fail to find the global optimum, or do not guarantee to have found the global optimum. 
A great variety of 
relaxations and solution methods to solve the OPF problem has been tested, 
including non-linear programs, piece-wise linearization, Lagrangian relaxations, genetic algorithms, and interior point methods. For examples, please see surveys \cite{Pandya2008, Low2014}. 
\inTR{In a comprehensive computational study \cite{castillo2013}, the staff of Federal Energy Regulatory Commission have tested the performance of five leading optimization solvers (Conopt, Ipopt, Knitro, Minos, and Snopt) on instances of up to 3,120 buses, but none of the solvers has produced a feasible solution in more than $82 \%$ of experiments and the best solution out of those known in more than $73 \%$ of experiments. 
A number of recent relaxations makes it possible to guarantee the global
optimality in some cases, though.}
A recent line of research proposed by Bai et al. \cite{Bai2008} applied semidefinite programming (SDP) to the OPF problem. Lavaei et al. \cite{lavaei2012zero, sojoudi2013semidefinite} then showed that the solution of the SDP is the global optimum, under some conditions. Several follow-up computational studies \cite{Jabr2012_2,Molzahn2011} increased the  dimension of SDP instances that can be solved, in practice.
\inTR{
\begin{table*}
\begin{tabularx}{\textwidth}{Xrrrrr}
\hline
Termination & Conopt & Ipopt & Knitro & Minos & Snopt \\ \hline
Exceeded time limit with an infeasible solution & 7.9\% & 22.2\% & 38.4\% & 0.0\% & 23.0\% \\
Exceeded time limit with a feasible solution & 6.5\% & 0.9\% & 6.7\% & 0.0\% & 14.4\% \\
Early termination with an infeasible solution & 0.0\% & 0.4\% & 0.6\% & 11.4\% & 0.0\% \\
Early termination with a feasible solution & 31.2\% & 0.6\% & 1.8\% & 2.3\% & 0.2\% \\
Normal termination with an infeasible solution & 10.5\% & 3.4\% & 0.0\% & 55.9\% & 1.6\% \\ \hline
Percentage of feasible & 81.6\% & 74.1\% & 61.0\% & 32.7\% & 75.4\% \\ \hline
Percentage of best seen & 43.8\% & 72.5\% & 52.5\% & 30.3\% & 60.8\% \\ \hline 
\end{tabularx}\\[3mm]
\caption{Performance of leading MINLP solvers on a benchmark collection of ACOPF instances, as reported by Castillo and O'Neill \cite{castillo2013}.
118--3120 bus instances, 100 runs of 20 minutes per instance and solver, Xeon E7458 at 2.4GHz with 64 GB RAM.}
\end{table*}
}

Although there has been no single formulation and solution approach suitable for all the various forms of OPF problems, 
many OPF formulations take the form of a polynomial programming (PP),
where the objective, equality constraints and inequality constraints are all given by 
multi-variate polynomials. 
Equality constraints typically include the power flow network equations and balance constraints. 
The inequality constraints often include active/reactive power generation limits, demand constraints, bus voltage limits, and branch flow limits. 
Using polynomial optimization, one can model the network more accurately, and 
obtain globally valid lower bounds and globally optimal solutions, under mild conditions.
Notably, one can use a wide variety of objective functions,
 and incorporate further constraints easily,
 without affecting convergence properties. 

For example, besides the minimization of power generation costs, 
other objectives can be formulated using PP, 
including minimization of power generation costs with unit commitment costs,
minimization of system losses, 
and 
maximization of power quality (minimizing voltage deviation).
Computationally, one uses a hierarchy of SDP relaxations of Lasserre \cite{Lasserre1}
 to convexify the PP problem. This approach was used for the OPF problem to improve the Lavei-Low bounds (See \cite{Molzahn2013,RTE2013}).
Unfortunately, the dimension of these relaxations grow rapidly with the size of the power system, 
 posing a major computational challenge.

In this paper, we present two techniques for tackling the OPF problem. 
\inTR{
The first one is generic and can be applied to polynomial programs of any structure. 
The second technique benefits from the structure of the underlying polynomial optimization problem.
}
The first technique uses ``cutting surfaces'' of Ghaddar et al. \cite{BissanThesis}, which are valid inequalities, generated dynamically upon violation at each step of the algorithm. 
Instead of increasing the degree of the non-negative certificates, as in Lasserre's hierarchy, the set of polynomial inequalities describing the feasible region of the polynomial program is changed in each iteration, while the degree of the polynomials is fixed. 
These valid inequalities yield stronger convexifications and hence tighter bounds than the Lavaei-Low \cite{lavaei2012zero} SDP relaxation.


The second technique uses the sparse hierarchy of SDP relaxations of Waki et al. \cite{WKKM}, which improves the tractability of the Lasserre's hierarchy by exploiting sparsity of the OPF problem. 
The relaxations are equivalent to the Lavaei-Low \cite{lavaei2012zero} SDP relaxation, where it is exact, and provide tighter relaxations, where it is not (i.e., as the level of the hierarchy increases). 
Further, we employ matrix completion techniques \cite{KKMY} to break down the largest SDP matrix at the price of introducing additional equality constraints and several smaller matrix inequalities,
to make the approach scale to power systems with thousands of buses.

Overall, the main contributions of the paper are:
\begin{itemize}
\item stronger convexifications for the OPF problem than those presented in the literature
 \item larger instances than those published in the literature are solved to global optimality by exploiting structured sparsity of the OPF problem
  \item proof of convergence of the sparse hierarchy of SDP relaxations for OPF.
\end{itemize}
Notably, either of the presented techniques improves upon the Lavaei-Low SDP relaxation, whenever the Lavaei-Low SDP relaxation does not provide the global optimum.  


%% file: OPF_Formulation.tex
\section{Optimal Power Flow Problem}\label{sec:opf}
We use the same notation as in \cite{lavaei2012zero} and \cite{Molzahn2011}. The topology of the power system $P = (N, E)$ is represented as an undirected graph, where each vertex $n \in N$ is called a ``bus'' and each edge $e \in E$ is called a ``branch''. 
We use $|N|$ to denote the number of buses and $|E|$ to denote the number of branches. Let $G \subseteq N$ be the set of generators and $E \subseteq N \times N$ be the set of all branches modeled as $\Pi$-equivalent circuits. 
The matrix $y \in \mathbb{R}^{|N|\times|N|}$ represents the network admittance matrix, whose sparsity pattern is the same as that of the adjacency matrix of $P$. $\bar{b}_{lm}$ is the value of the shunt element at branch $(l,m) \in E $ and $g_{lm}+jb_{lm}$ is the series admittance on a branch $(l,m)$. Let $S^d_k=P^d_k + jQ^d_k$ be the active and reactive load (demand) at each bus $k \in N$ and $P^g_k + jQ^g_k$ represent the apparent power of the generator at bus $k \in G$. Define $V_k =\Re{V_k}+ j \Im{V_k}$ as the voltage at each bus $k \in N$ and $S_{lm}=P_{lm}+jQ_{lm}$ as the apparent power flow on the line $(l,m) \in E$. The edge set $L \subseteq E$ contains the branches $ (l,m)$ such that the apparent power flow limit is less than a certain given tolerance $\varepsilon$. \inTR{When the output of at most one generator is a variable and the objective function is related to losses, the problem is known as the minimum loss power flow. When the output of more than one generator is variable, within given bounds, and the objective is related to the cost of generation, the problem is known as the optimal power flow problem.} 


\subsection{Formulation}
We focus on the rectangular power-voltage formulation, where 
\begin{itemize}
\item $P_k^{\min}$ and $P_k^{\max}$ are the limits on active generation capacity at bus $k$, where $P_k^{\min}=P_k^{\max} =0$ for all $k \in N / G$.
\item $Q_k^{\min}$ and $Q_k^{\max}$ are the limits on reactive generation capacity at bus $k$, where $Q_k^{\min}=Q_k^{\max} =0$ for all $k\in N / G$. 
\item $V_k^{\min}$ and $V_k^{\max}$ are the limits on the absolute value of the voltage at a given bus $k$. 
\item $S_{lm}^{\max}$ is the limit on the absolute value of the apparent power of a branch $(l,m)\in L$.
\end{itemize}

Let $e_k$ be the $k^{th}$ standard basis vector in $\mathbb{R}^{|N|}$, similar to \cite{lavaei2012zero}, the following matrices are defined
\begin{align*}
y_k&=e_ke_k^T y, \\
y_{lm}&=(j\frac{\bar{b}_{lm}}{2}+g_{lm}+jb_{lm})e_le_l^T - (g_{lm}+jb_{lm})e_le_m^T, \\
Y_k&= \frac{1}{2} \left[ \begin{matrix} \Re(y_k  + y_k^T)  &\Im(y_k^T -y_k)  \\
   \Im(y_k  - y_k^T))& \Re(y_k  + y_k^T) \end{matrix}  \right], \\
\bar{Y}_k&= -\frac{1}{2} \left[ \begin{matrix} \Im(y_k  + y_k^T)& \Re(y_k -y_k^T) \\
   \Re(y_k^T  - y_k) & \Im(y_k  + y_k^T) \end{matrix}  \right], \\
   M_k&=  \left[ \begin{matrix}e_ke_k^T & 0  \\
   0 & e_ke_k^T \end{matrix}  \right], \\
   Y_{lm}&= \frac{1}{2} \left[ \begin{matrix} \Re(y_{lm}  + y_{lm}^T)& \Im(y_{lm}^T -y_{lm}) \\
   \Im(y_{lm} - y_{lm}^T) & \Re(y_{lm}  + y_{lm}^T) \end{matrix}  \right],\\
   \bar{Y}_{lm}&=- \frac{1}{2} \left[ \begin{matrix} \Im(y_{lm}  + y_{lm}^T)& \Re(y_{lm}^T -y_{lm}) \\
   \Re(y_{lm}^T - y_{lm}) & \Im(y_{lm}  + y_{lm}^T) \end{matrix}  \right].
  \end{align*}

Let $x$ be a vector of variables defined as $x: =[\Re{V}_k \quad \Im{V}_k ]^T$, and let the cost of power generation be  $\sum_{k\in G} f_k(P_k^g)$ where $f_k(P_k^g)= c^2_k (P_k^g)^2 + c^1_k P_k^g + c^0_k$, with $c^2_k, c^1_k, c^0_k$ non-negative. The classical OPF problem 
can be written as a polynomial optimization problem of degree 4,
{\small\begin{align}
\min &\sum_{k\in G} f_k(x) \tag*{[OP$_4$]}  
\\
\stt &P_k^{\min} \leq \text{tr}(Y_kxx^T)+P_k^d \leq P_k^{\max} \inTR{ \: \quad \qquad \qquad \forall i \in N} \notag \\ 
& Q_k^{\min}\leq \text{tr}(\bar{Y}_kxx^T)+Q_k^d  \leq Q_k^{\max} \inTR{ \qquad \qquad \qquad\forall i \in N } \notag \\ 
&(V_k^{\min})^2 \leq \text{tr}(M_kxx^T) \leq  (V_k^{\max})^2 \inTR{ \: \: \:  \quad\qquad\qquad \forall i \in N} \notag \\ 
&(\text{tr}(Y_{lm}xx^T))^2+(\text{tr}(\bar{Y}_{lm}xx^T))^2 \leq (S_{lm}^{\max})^2 \inTR{ \qquad \forall (l,m) \in L} \notag  
\end{align}}
 The objective function often is the cost of power generation where {\small $$f_k(x):=\left(c^2_k(P_k^d + \text{tr}(Y_kxx^T))^2+c^1_k(P_k^d + \text{tr}(Y_kxx^T))+c^0_k\right)$$}. The constraints, in turn, 
impose a limitation on the active and reactive power, 
restrict the voltage on a given bus, 
and limit the apparent power flow at each end of a given line. 
By defining variable $W=xx^T$,
\cite{lavaei2012zero} reformulates the problem as a rank-constrained problem.
Subsequently, one can drop the rank constraint to obtain the SDP relaxation [OP-SDP] as in \cite{lavaei2012zero}. 

%% file: PP_Dense.tex
\section{Polynomial Programming Approach}\label{sec:pp}
The OPF problem is a particular case of a polynomial optimization problem of the form: \inTR{whose objective and constraints are multivariate polynomials. In general, a polynomial program has the following form:}
\begin{align}
\min \quad& f(x) \notag \\
\mbox{s.t.  }& g_i(x) \geq 0 \qquad i=\{1,\dots,m\} \tag*{[PP]}
\end{align}
Motivated by the seminal work of Lasserre \cite{Lasserre1}, there has been a lot of recent research activity to devise solution schemes to solve polynomial optimization problems. The main idea of these schemes is based on applying representation theorems from algebraic geometry to characterize the set of polynomials that are nonnegative on a given domain. 
Given $S\subseteq \mathbb{R}^n,$ define $\mathcal{P}_d(S)$ to be the cone of polynomials of degree at most $d$ that are non-negative over $S$.  We use $\Sigma_d$ to denote the cone of polynomials of degree at most $d$ that are sum-of-squares of polynomials. 
Using $\mathcal{G}=\{g_i(x): i=1,\dots,m \}$ and denoting $S_{\mathcal{G}}=\{x \in \mathbb{R}^n :  g(x) \geq 0, \; \forall g \in \mathcal{G}\}$ the basic closed semi-algebraic set defined by $\mathcal{G}$, we can rephrase [PP] as
\begin{align}
\max \quad & \varphi & \mbox{s.t.  }& f(x)-\varphi  \geq 0 \quad \forall \: x \in S_G, \notag \\
= \max \quad & \varphi & \mbox{s.t.  } &f(x)-\varphi \in \mathcal{P}_d(S_{\mathcal{G}}). \tag*{[PP-D]}
\end{align}

Although [PP-D] is a conic problem, it is not known how to optimize over the cone $\mathcal{P}_d(S_{\mathcal{G}})$ efficiently. Lasserre \cite{Lasserre1} introduced a hierarchy of SDP relaxations corresponding to liftings of polynomial problems into higher dimensions. In the hierarchy of SDP relaxations, one convexifies the problem, obtains progressively stronger relaxations, but the size of the SDP instances soon becomes computationally challenging. 
Under assumptions slightly stronger than compactness, the optimal values of these problems converge to the global optimal value of the original problem, [PP]. 

The approximation of $\mathcal{P}_d(S_{\mathcal{G}})$
used in \cite{Lasserre1} is the cone $\mathcal{{K}}^r_{\mathcal{G}}$, where
{\small \begin{align}
\mathcal{{K}}^{r}_{\mathcal{G}} =  \Sos_{r}+\sum_{i=1}^{m}g_i(x) \Sos_{r-\deg(g_i)}, \label{eqn:K} 
\end{align}}
and $r \geq d$. 
The corresponding optimization problem over $S$ can be written as:
{\small {\begin{align}
\max_{\varphi, \sigma_i(x)} \:  & \varphi \tag*{[PP-H$_r$]$^*$} \label{eq-Lass} \\
\mbox{s.t. } & f(x)-\varphi= \sigma_0(x)+\sum_{i=1}^{m} \sigma_i(x)g_i(x) \notag\\
&\sigma_0(x) \in \Sos_{r}, \: \sigma_i(x) \in \Sos_{r-\deg(g_i)}. \notag  
\end{align}}}
\ref{eq-Lass} can be reformulated as a semidefinite optimization problem. We denote the dual of [PP-H$_r$]$^*$ by [PP-H$_r$]. The computational cost of the problem clearly depends on both the degree of the polynomials, $r$, and the dimension of the problem. The number of constraints can be large, especially when many variables and high-degree polynomials are used. Based on the described approach, Molzahn and Hiskens \cite{Molzahn2013} and Josz et al. \cite{RTE2013} used [OP$_4$] and applied Lasserre's hierarchy to obtain global optimality on instances with up to 5 and 10 buses respectively, where Lavaei-Low is not globally optimal. 

\subsection{Relationship with Lavaei-Low Formulation} \label{sec:LLD}
In this work, instead of starting with [OP$_4$] and 
applying the hierarchy [PP-H$_r$]$^*$, we reduce the OPF problem to a polynomial program of degree 2:
{\small  {\begin{align*}
\min &\sum_{k\in G} \left(c^2_k(P^g_k)^2+c^1_k(P_k^d + \text{tr}(Y_kxx^T))+c^0_k\right) \notag \\
&P_k^{\min} \leq \text{tr}(Y_kxx^T)+P_k^d \leq P_k^{\max} \tag*{[OP$_2$]}\\
& Q_k^{\min}\leq \text{tr}(\bar{Y}_kxx^T)+Q_k^d\leq Q_k^{\max} \notag \\
&(V_k^{\min})^2 \leq \text{tr}(M_kxx^T) \leq  (V_k^{\max})^2 \notag \\
&P_{lm}^2+Q_{lm}^2 \leq (S_{lm}^{\max})^2 \notag \\
& P^g_k = \text{tr}(Y_kxx^T) +P_k^d \notag \\
& P_{lm} = \text{tr}(Y_{lm}xx^T) \notag \\
& Q_{lm} = \text{tr}(\bar{Y}_{lm}xx^T). \notag 
\end{align*}}}
\begin{theorem}
\label{thm:equivalence}
The first level of the hierarchy for [OP$_2$], [OP$_2$-H$_1$]$^*$, is equivalent to the dual of [OP-SDP], i.e., Optimization 4 in \cite{lavaei2012zero}.
\end{theorem}
Hence, the first level of the hierarchy for [OP$_2$] (i.e., $r=2$), [OP$_2$-H$_1$]$^*$ provides the same bound as [OP-SDP] while the first level of the hierarchy for [OP$_4$] (i.e., $r=4$), provides a bound that is at least as good as [OP-SDP].


%% file: PP_Cuts.tex
\subsection{Inequality Generation Approach}\label{sec:digs}
Realising that ACOPF can be modeled as a polynomial program, we aim to tackle the problem using methods recently developed in polynomial optimization. 
The first method, discussed in this section, is a dynamic approach, which generates valid, but violated inequalities at each step of the algorithm. The idea is based on the work of Ghaddar et al. \cite{Ghaddar2011}, who proposed the dynamic inequality generation scheme (DIGS) for general PP. 
In DIGS, 
the current solution is used to generate polynomial inequalities that are valid on the feasible region of the PP problem. 
This iterative scheme makes it possible to generate improving approximations without growing the degree of the certificates involved, and hence the size of the SDP problem.

In this section, we use the first level of the relaxation of [OP$_2$], i.e. [OP$_2$-H$_1$], and add valid quadratic inequalities of the form $p(x) \geq 0$. 
The polynomial $p(x)$ needs to be a valid inequality and at the same time improve on the bound of the relaxation. This can be translated as $ p \in \mathcal{P}_d(S) \setminus \mathcal{K}^d_{\mathcal{G}}$, where $d=2$, the degree of [OP$_2$], in this case. 
The iterative scheme can be summarized as follows:
\begin{itemize}\label{ProcedureCut}
 \item Start with ${\mathcal{G}}_0={\mathcal{G}}$
 \item Given ${\mathcal{G}}_i$ let $p_{i}  \in \mathcal{P}_d(S_{\mathcal{G}}) \setminus \mathcal{K}^d_{{\mathcal{G}}_{i}},$ define ${\mathcal{G}}_{i+1}={\mathcal{G}}_{i} \cup \{ p_{i} \}.$ \\
\end{itemize}
To be able to generate a polynomial $p(x)$, the scheme consists of a master problem and a subproblem. The master problem is of the same form as \ref{eq-Lass} with the hierarchy level (i.e., $r$) being fixed to $d$:
 \begin{align}
\max \: & \varphi & \mbox{ s.t. } & f(x)-\varphi \in \mathcal{{K}}^{d}_{\mathcal{G}}, \tag*{[PP-M]}
 \end{align}
 where $ \mathcal{{K}}^{d}_{\mathcal{G}}$ is as defined in \eqref{eqn:K} with $r$ fixed to $d$. 
The master problem provides lower bounds. 
The subproblem uses the optimal dual information from the master problem, $Y$, to generate polynomial inequalities that are valid on the feasible region: 
 \begin{align}
\min_p \: & \left<p, Y \right>  
 &  \mbox{ s.t. } & p \in \mathcal{K}_{\mathcal{G}}^{{d}+2}\cap \mathbf{R}_{d}[x]. \tag*{[PP-S]}
 \end{align}
These valid inequalities are then incorporated into the master problem, to construct new non-negativity certificates, 
obtaining better approximations of the OPF.
The iterative scheme terminates when the objective function of the subproblem is sufficiently close to 0 \cite{BissanThesis}. 

Considering [OP$_2$-H$_1$] is equivalent to the Lavaei-Low SDP relaxation,
as seen in Section \ref{sec:LLD}, one can improve on the Lavaei-Low bound
by adding valid inequalities to [OP$_2$]. In some cases, one can obtain the global optimum. 
As opposed to the approach proposed in \cite{RTE2013,Molzahn2013}, where the hierarchy level, $r$, is increased \eqref{eq-Lass}, in this case, $r$ is fixed to $d$. 
That is: Instead of increasing the degree of the non-negative certificates, 
the degree of the polynomials is fixed and the set of polynomial inequalities describing the feasible region of the polynomial program is increased. 
Valid inequalities are used to construct new certificates that provide better approximations and hence provide stronger convexification at each iteration. 
Consequently, the relaxation is improving at each stage of the algorithm and hence better bounds are obtained at each iteration, while sizes of the positive semidefinite matrices and the numbers of constraints can be significantly lower as compared to \ref{eq-Lass}.

%% file: PP_Sparse.tex
\subsection{Exploiting OPF Structure} \label{sec:sparse} 
The current scalability of state-of-the-art SDP solvers limits the tractability of the Lasserre hierarchy even for medium-scale polynomial programs. 

One approach to improve the tractability of the Lasserre hierarchy is to exploit {\it correlative} sparsity of a polynomial optimization problem [PP] of dimension $n$ due to Waki et al. \cite{WKKM}, which can be represented by the symbolic $n\times n$ {\it correlative sparsity pattern matrix} $R$,  defined by
$$R_{ij} = \begin{cases} 
\star& \text{for } i=j\\ 
\star& \text{for } x_i, x_j \text{ in the same monomial of } f \\ 
\star& \text{for } x_i, x_j \text { in the same constraint } g_k \\ 
0 & \text {otherwise},
\end{cases} $$

and its associated adjacency graph $G$, the {\it correlative sparsity pattern graph}. Let $\{I_k \}_{k=1}^p$ be the set of maximal cliques of a chordal extension of $G$ following the construction in \cite{WKKM}, i.e. $I_k \subset \{ 1,\ldots,n \}$. Given that a chordal extension of arbitrary graphs is not unique, it is important to choose a chordal extension which can be computed efficiently while keeping the number of additional edges as small as possible, since the size of matrix inequalities in the sparse hierarchy is determined by the cardinality of the maximal cliques $I_1,\ldots, I_p$. Note, that the ordering $O$ applied to $R$ determines the chordal extension of $R$, and hence the number and cardinality of the maximal cliques $\{I_k \}_{k=1}^p$. The problem of minimising the size of the sparse hierarchy of SDP
relaxations for [PP] is therefore equivalent to finding the ordering which results in a symbolic Cholesky factorisation with the minimal number of fill-ins, or finding the chordal extension of $G$ with the minimal number of edges added.  While the problem of finding the minimal chordal extension is NP-hard, a number of heuristics for orderings have been proposed that aim to keep the number of fill-ins in the symbolic Cholesky factorisation such as the symmetric approximate minimum degree ordering. The sparse approximation of $\mathcal{P}_d(S)$ is $\mathcal{{K}}^{r}_{\mathcal{G}}(I)$, given by
\begin{align*}
  \mathcal{{K}}^{r}_{\mathcal{G}}(I) = \sum_{k=1}^p \left(\Sigma_r(I_k) +  \sum_{j\in J_k} g_j \Sigma_{r-\deg(g_j)}(I_k)\right),
\end{align*}
where $\Sigma_d(I_k)$ is the set of all sum-of-squares polynomials of degree up to $d$ supported on $I_k$ and $(J_1,\ldots, J_p)$ is a partitioning of the set of polynomials $\{g_j\}_j$ defining $S$ such that for every $j$ in $J_k$, the corresponding $g_j$ is supported on $I_k$. The support $ I \subset \{1,\ldots,n\} $ of a polynomial contains the indices $i$ of terms $x_i$ which occur in one of the monomials of the polynomial.  The sparse hierarchy of SDP relaxations is then given by 
{\small\begin{align}
&\max_{\varphi,\sigma_k(x), \sigma_{r,k}(x)} \:   \varphi \tag*{[PP-SH$_r$]$^*$} \label{eq-Lass-sparse}  \\ 
&\mbox{s.t. } f(x)-\varphi=  \sum_{k=1}^p \left(\sigma_{k}(x) +  \sum_{j\in J_k} g_j(x) \sigma_{j,k}(x)\right) \notag \\
&\quad \sigma_{k}\in\Sos_r((I_k)), \sigma_{j, k} \in \Sos_{r-\deg(g_j)}(I_k). \notag 
\end{align}}
We denote the dual of [PP-SH$_r$]$^*$ by [PP-SH$_r$]. In the case that $R$ is sparse, i.e., $\mid I_k \mid \ll n$, then the resulting matrix variables are of size ${\mid I_k \mid + r \choose r }$, instead of  ${ n + r \choose r }$. While \ref{eq-Lass-sparse} provides a weaker relaxation to [PP] than \ref{eq-Lass} for a fixed relaxation order $r$ in general, the asymptotic convergence result for the dense hierarchy extends to the sparse case:
\begin{assumption}
\label{as:Lasserre}
Let $S$ denote the feasible set of a problem of form [PP]. Let $\{I_k \}_{k}$ denote the $p$ maximal cliques of a chordal extension of the sparsity pattern graph of the [PP].
\begin{enumerate}
\item[(i) ]
Then, there is a $M>0$ such that $\parallel x \parallel_{\infty} < M$ for all $x\in S$.
\item[(ii) ]
Ordering conditions for index sets as in Assumption 3.2 (i) and (ii) in \cite{Lasserre2006}.
\item[(iii) ]
Running-intersection-property, c.f. \cite{Lasserre2006}, holds for $\{I_k\}_k$.
\end{enumerate}
\end{assumption}

\begin{remark}
Note, that if $S$ is compact, it is easy to add up to $p$ redundant quadratic inequality to the definition of $S$, s.t. Assumption \ref{as:Lasserre} (i) is satisfied. (ii) can be satisfied by construction and re-ordering of the sets $\{I_k \}_{k}$. The running-intersection-property is satisfied for the maximal cliques of a choral graph, as pointed out in \cite{KojimaMuramatsu2009}. Thus, Assumption \ref{as:Lasserre} is satisfied for both, [OP$_2$] and [OP$_4$]. 
\end{remark}

Now, we can formulate the convergence result.

\begin{proposition}[Asymptotic Convergence]
\label{thm:convergent}
If Assumption \ref{as:Lasserre} holds for the feasible sets of [OP$_2$] and [OP$_4$] respectively, then for the sparse hierarchy [OP$_2$-SH$_r$]$^*$ for [OP$_2$] and [OP$_4$-SH$_r$]$^*$ for [OP$_4$] and their respective duals the following holds: 
\begin{enumerate}
\item[(a) ]
$\inf \text{[OP$_2$-SH$_r$]} \nearrow \min \text{([OP$_2$])} \text{ as } r\rightarrow\infty,$\\
$\inf \text{[OP$_4$-SH$_r$]} \nearrow \min \text{([OP$_4$])} \text{ as } r\rightarrow\infty.$
\item[(b) ]
$\sup \text{[OP$_2$-SH$_r$]$^*$} \nearrow \min \text{([OP$_2$])} \: \text{as} \: r\rightarrow\infty,$\\
$\sup \text{[OP$_4$-SH$_r$]$^*$} \nearrow \min \text{([OP$_4$])} \: \text{as} \: r\rightarrow\infty.$
\item[(c) ]
(i) If the interior of the feasible set of [OP$_4$] is nonempty, there is no duality gap between [OP$_4$-SH$_r$] and [OP$_4$-SH$_r$]$^*$.\\
(ii) There is no duality gap between [OP$_2$-SH$_1$] and [OP$_2$-SH$_1$]$^*$.
\item[(d) ] If [OP$_2$] has a unique global minimizer, $x^*$, then as $r$ tends to infinity the components of the optimal solution of [OP$_2$-SH$_r$] corresponding to the linear terms converge to $x^*$ (an analogous result holds for [OP$_4$]).
\end{enumerate}
\end{proposition}
\begin{proof}
(a), (c) and (d) are a Corollary of Theorem~3.6 of Lasserre \cite{Lasserre2006}, (b) (i) follows from Theorem 5 of \cite{KojimaMuramatsu2009}, (b) (ii) from note in \cite{WKKM} on primal and dual sparse SDP relaxation of order 1 for quadratic optimization problems.
\end{proof}

Moreover, for [OP$_2$] the following proposition holds.

\begin{proposition}
The sparse SDP relaxation [OP$_2$-SH$_1$]$^*$ of order one for [OP$_2$] is equivalent to the first order relaxation of the dense Lasserre hierarchy [OP$_2$-H$_1$]$^*$ for [OP$_2$] and the Lavaei-Low dual sdp relaxation.
\end{proposition}

\begin{proof}
Follows from the fact that sparse and dense SDP relaxation of order 1 are equivalent for non-convex quadratic optimization problems, as proven in Section 4.5 of Waki et al. \cite{WKKM} and the Theorem \ref{thm:equivalence}.
\end{proof}

\begin{remark}
For a fixed order $r$, the sparse hierarchy [OP$_2$-SH$_1$] has $O(\kappa^{2r})$ variables, where $\kappa$ is the maximum number of variables appearing in the objective or an inequality constraint of [PP]. The largest matrix inequality is of size $O(\kappa^r)$. This is in contrast to $O(n^{2r})$ variables and matrix variables of size $O(n^r)$ in the dense hierarchy [OP$_2$-H$_1$]. In case the [PP] is very sparse, i.e., $\kappa \ll n$, the size of the sparse hierarchy is vastly smaller then the dense one.
\end{remark}

%% file: results.tex
\section{Numerical Results} \label{sec:results}
In order to illustrate the performance of the approaches, we improve the relaxation [OP$_2$-H$_1$]$^*$ iteratively using the inequality generation scheme (we refer to as DIGS), in some cases to proven optimality. Next, we exploit the sparsity of the polynomial program to solve higher order relaxations, proving global optimality for further instances.\footnote{The two techniques are implemented in MATLAB running on a PC with a 3.5Ghz processor, running Red Hat Linux. The DIGS approach is implemented using APPS \cite{BissanThesis}. SparseColO was used for exploiting sparsity in the master and the subproblem. Exploiting sparsity of polynomial program [OP$_4$] is done using SparsePoP \cite{WKKM2}. To solve the resulting SDP relaxation for both approaches, SeDuMi \cite{Sturm} is used as the SDP solver. All test instances are taken from \cite{Bukhsh2013} and \cite{powerwebsite}. For the polish instances the formulation in \cite{Molzahn2011} is used to take into account multiple generators and transformers with off-nominal voltage ratios and phase shifts.} Figure \ref{fig:opfrelaxations} summarizes the formulations and relaxations discussed in this paper.
\begin{figure}[tp]\centering
\withdvi{\includegraphics[trim=0.5cm 2.5cm 8cm 4cm, clip=true, width = 0.5\textwidth]{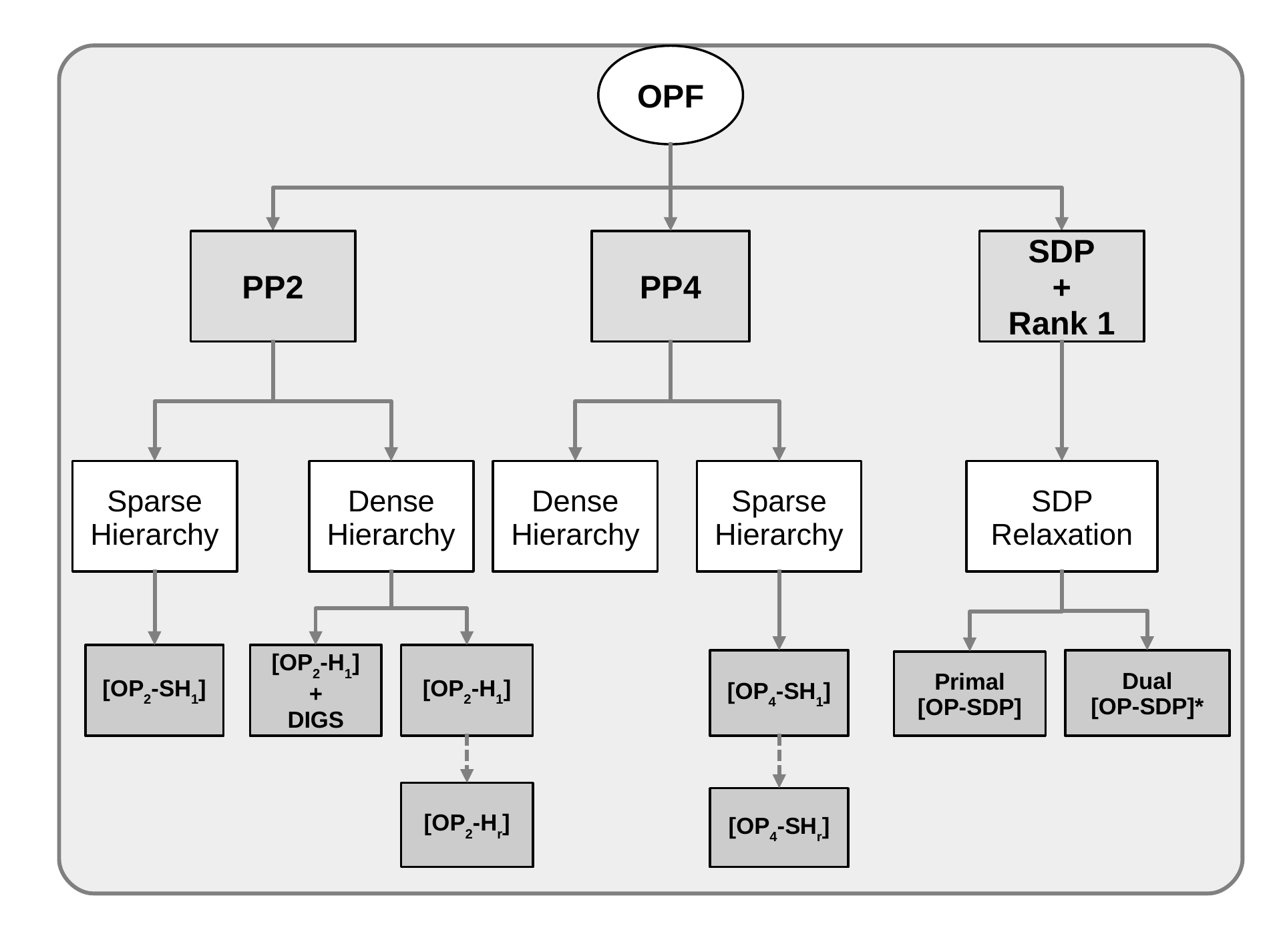}}
\caption{The different formulations and their relaxations.
}\label{fig:opfrelaxations}
\end{figure}

\subsection{Small-scale Instances}
First, we use three small test cases, devised by Bukhsh et al. \cite{Bukhsh2013} and Lesieutre et al. \cite{Lesieutre2011} such that the Lavaei-Low SDP relaxation is not optimal. In Tables~\ref{tb:WB2}--\ref{tb:WB5}, bold entries indicate that our approaches prove global optimality.

\begin{example}\label{ex-WB2}
The first example is an instance with two buses \cite{Bukhsh2013}, where the maximum voltage on the second bus varies from 0.976 to 1.028 (see Table \ref{tb:WB2}). From Figure \ref{fig:WB2}, it can be seen that after adding one inequality (i.e., one iteration of DIGS), the Lavaei-Low bound improves significantly. The optimal value is obtained in 8 iterations. For none of the 9 instances, MATPOWER converged and the solution MATPOWER found was far from the optimal. For example, for  $V_2^{\max}$=1.022, the value of 713.27 is reported. 
\begin{table}[t]
\caption{WB2 computational results.}
\setlength{\tabcolsep}{1pt}
{\label{tb:WB2}
\begin{tabular}{l|ccc||cccccc}
&\multicolumn{3}{c}{DIGS}&\multicolumn{2}{c}{[OP$_2$-SH$_1$]}&\multicolumn{2}{c}{[OP$_4$-SH$_1$]} &\multicolumn{2}{c}{[OP$_4$-SH$_2$]}\\ \hline
$V_2^{\max}$&  Iter $s$& $s$ & Time&Bound& Time&Bound& Time&Bound& Time\\
\hline\hline
0.976	&		\textbf{905.76}	&	1	&	0.9	&	\textbf{905.76}	&	0.2	&	\textbf{905.76}	&	0.4	&	&	\\	
0.983	&		\textbf{905.73}	&	6	&	5.1	&	903.12	&	0.2	&	\textbf{905.73}	&	1.8	&	&	\\	
0.989	&		\textbf{905.73}	&	6	&	4.3	&	900.84	&	0.1	&	905.72	&	1.7	&	\textbf{905.73}	&	1.8\\
0.996	&		\textbf{905.73}	&	6	&	4.6	&	898.17	&	0.2	&	905.73	&	1.4	&	\textbf{905.73}	&	1.6\\
1.002	&		\textbf{905.73}	&	6	&	4.8	&	895.86	&	0.1	&	905.72	&	1.8	&	\textbf{905.73}	&	1.5\\
1.009	&		\textbf{905.73}	&	8	&	6.4	&	893.16	&	0.2	&	905.71	&	1.9	&	\textbf{905.73}	&	0.6\\
1.015	&		\textbf{905.73}	&	6	&	4.7	&	890.82	&	0.1	&	905.71	&	0.8	&	\textbf{905.73}	&	0.6\\
1.022	&		\textbf{905.73}	&	8	&	6.5	&	888.08	&	0.1	&	905.71	&	2.6	&	\textbf{905.73}	&	1.7\\
1.028	&		\textbf{905.73}	&	8	&	5.1	&	885.71	&	0.1	&	904.59	&	0.8	&	\textbf{905.73}	&	0.8\\
\hline
\end{tabular}}
\end{table}
\begin{figure}[tp]\centering
\withdvi{\includegraphics[trim=1cm 7.5cm 1cm 8.5cm, clip=true, width = 0.5\textwidth]{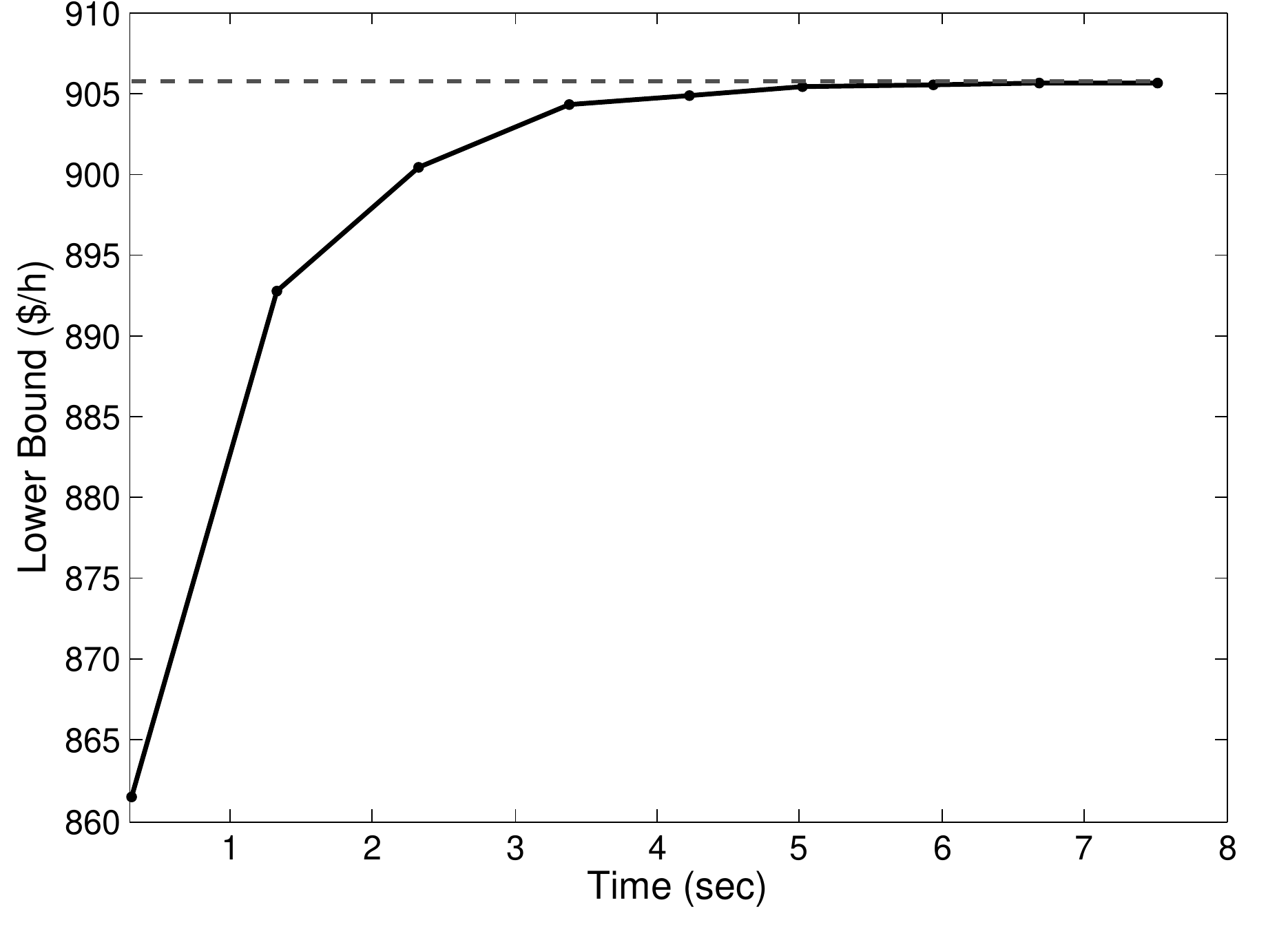}}
\caption{WB2 Bounds for $V_2^{\max}$=1.022.}\label{fig:WB2}
\end{figure}
\end{example}

\begin{example}\label{ex-LMBM3}
The second example is LMBM3 with 3 buses \cite{Lesieutre2011}. The results from Table \ref{tb:LMBM3} indicate that both approaches are successful in providing the optimal solution of these problems when Lavaei-Low relaxation is not exact. In this example, MATPOWER reported the optimal solution for all 10 instances.
\begin{table}[tp]
\caption{LMBM3 computational results.}
\setlength{\tabcolsep}{3pt}
{\label{tb:LMBM3}
\begin{tabular}{l|ccc||cccc}
&\multicolumn{3}{c}{DIGS}&\multicolumn{2}{c}{[OP$_2$-SH$_1$]}&\multicolumn{2}{c}{[OP$_4$-SH$_1$]}\\ \hline
$S_{23}^{\max}$&  Iter $s$& $s$ & Time&Bound& Time&Bound& Time\\
\hline\hline
28.35&\textbf{10294.88}&7&13.3	&6307.97	& 0.2	&\textbf{10294.88}	&	1.0\\
31.16&\textbf{8179.99}&6&11.2	&6206.78	& 0.2	&\textbf{8179.99}	&	0.7\\
33.96&\textbf{7414.94}&5&19.2	&6119.71	& 0.2	&\textbf{7414.94}	&	0.8\\
36.77&\textbf{6895.19}&5&19.5	&6045.33	& 0.3	&\textbf{6895.19}	&	0.7\\
39.57&\textbf{6516.17}&5&19.8	&5979.38	& 0.2	&\textbf{6516.17}	&	0.7\\
42.38&\textbf{6233.31}&5&18.1	&5919.12	& 0.2	&\textbf{6233.31}	&	0.7\\
45.18&\textbf{6027.07}&5&19.3	&5866.68	& 0.1	&\textbf{6027.07}	&	0.8\\
47.99&\textbf{5882.67}&3&12.1	&5819.02	& 0.2	&\textbf{5882.67}	&	0.7\\
50.79&\textbf{5792.02}&2&9.2	&5779.34	& 0.3	&\textbf{5792.02}	&	0.7\\
53.60&\textbf{5745.04}&1&0.7	&5745.04	& 0.2	&\textbf{5745.04}	&	0.8\\
\hline
\end{tabular}
}
\end{table}
 \end{example}
\begin{example}\label{ex-WB5}
The last example consists of 5 buses \cite{Bukhsh2013}. The results presented in Table \ref{tb:WB5} are consistent with the previous two examples. Applying DIGS and exploiting sparsity solved all test cases to optimality. MATPOWER provided the optimal solution for all 8 instances.
\begin{table}[tp]
\caption{WB5 computational results.}
\setlength{\tabcolsep}{3pt}
{\label{tb:WB5}
\begin{tabular}{l|ccc||cccc}
&\multicolumn{3}{c}{DIGS}&\multicolumn{2}{c}{[OP$_2$-SH$_1$]}&\multicolumn{2}{c}{[OP$_4$-SH$_1$]}\\ \hline
$Q_5^{\min}$& Iter $s$& $s$ & Time&Bound& Time&Bound& Time\\
\hline\hline
-20.51	&		\textbf{1146.48}	&	3	&	28.4	&	954.82	&	0.3	&	\textbf{1146.48}	&	25.5\\
-10.22	&		\textbf{1209.11}	&	4	&	32.6	&	963.83	&	0.3	&	\textbf{1209.11}	&	17.1\\
0.07	&		\textbf{1267.79}	&	5	&	49.0	&	972.80	&	0.2	&	\textbf{1267.44}	&	26.3\\
10.36	&		\textbf{1323.86}	&	5	&	49.4	&	981.89	&	0.4	&	\textbf{1323.86}	&	20.9\\
20.65	&		\textbf{1377.97}	&	4	&	39.1	&	990.95	&	0.2	&	\textbf{1377.97}	&	15.4\\
30.94	&		\textbf{1430.54}	&	4	&	40.1	&	1005.13	&	0.3	&	\textbf{1430.54}	&	20.9\\
41.23	&		\textbf{1481.81}	&	5	&	49.6	&	1033.07	&	0.3	&	\textbf{1481.81}	&	14.6\\
51.52	&		\textbf{1531.97}	&	5	&	49.7	&	1070.39	&	0.5	&	\textbf{1531.97}	&	18.2\\
\hline
\end{tabular}
}
\end{table}
 
\end{example}
\begin{table*}[t!]
\centering
\caption{Computational results for IEEE and Polish network instances.}
\setlength{\tabcolsep}{3pt}
{\label{tb:SP}
\begin{tabular}{l|r||rrr||rrr}
&MATPOWER&\multicolumn{3}{c||}{[OP$_2$-H$_1$]$^*$ + SparseColO }&\multicolumn{3}{c}{[OP$_4$-SH$_1$]}\\
Instance& Objective &Bound &Dim & Time& Bound & Dim & Time\\
\hline\hline
case9mod& 4267.07&2753.23 & 588$\times$168 & 0.6 & \textbf{3087.89}& 1792$\times$14847& 17.5\\
case14mod& 7806.10&7792.72 & 888$\times$94 & 0.9  & 7991.07&7508$\times$66740 & 904.2\\
case30mod& 623.01&576.89 & 4706$\times$684 & 3.8 &\textbf{578.56} & 36258$\times$49164& 13740.0\\
case39& 41864.18&41862.08 & 7282$\times$758 & 2.2 & \textbf{41864.18} & 26076$\times$215772&  4359.1\\
case57& 41737.79&41737.79 & 13366$\times$356& 3.2 &* & *\\
case118& 129660.69&129654.62 & 56620$\times$816& 6.1&* & *\\
case300& 719725.08&719711.63 & 362076$\times$1938 & 13.6 &* & *\\
case2383wp &  1.869$\times10^6$& 1.814$\times10^6$ & 22778705$\times$47975 & 3731.5  &* & *\\ 
case2736sp &1.308$\times10^6$ & 1.307$\times10^6$ & 30019740$\times$57408 & 3502.2  &* & *\\ 
\hline
\end{tabular}}
\end{table*}
\subsection{Large-scale Instances}
Next, we consider medium- and large-scale instances distributed with MATPOWER \cite{powerwebsite}. Table \ref{tb:SP} presents MATPOWER objective function value in addition to computational results for [OP$_2$-H$_1$]$^*$ with SparseColO and [OP$_4$-SH$_1$] using SparsePoP. [OP$_2$-H$_1$]$^*$ captures the Lavaei-Low dual relaxation and obtains the same bounds and has similar computational performance. The computational time of [OP$_2$-H$_1$]$^*$ can be significantly improved using SparseColO which utilizes domain-space sparsity of a semidefinite matrix variable and range-space sparsity of a linear matrix inequality constraint. It is used as a preprocessor, which reduces the dimension of matrix variables in an SDP relaxation before applying SeDuMi. For instances larger than 39 buses only [OP$_2$-H$_1$]$^*$ can be solved, as [OP$_4$-SH$_1$] becomes computationally expensive for SparsePoP. Using DIGS, optimality of case9mod is proven in 3 hours and 7 iterations. For case14mod, DIGS performed 2 iterations within 5 hours and improved upon the Lavaei-Low bound. For instances up to 2736 buses one can solve [OP$_2$-H$_1$]$^*$ within an hour, but the generation of cutting surfaces becomes too consuming.